\documentclass{amsart}
\usepackage{amsmath,amssymb,amsthm}
\usepackage[noend]{algorithm2e}
\usepackage[colorlinks=true,citecolor=blue]{hyperref}

\theoremstyle{plain}
\newtheorem{lemma}{Lemma}
\newtheorem{prop}[lemma]{Proposition}
\newtheorem{theorem}[lemma]{Theorem}
\newtheorem*{theorem*}{Theorem}
\newtheorem{definition}[lemma]{Definition}

\numberwithin{lemma}{section}
\numberwithin{equation}{section}

\theoremstyle{remark}

\newtheorem{rmk}[equation]{Remark}

\newcommand{\vareps}{\varepsilon}

\newcommand{\OO}{\mathcal{O}}
\newcommand{\QQ}{\mathbb{Q}}
\newcommand{\ZZ}{\mathbb{Z}}

\newcommand{\p}{\ell}
\newcommand{\eps}{\varepsilon_\p}

\newcommand{\kro}[2]{\left( \frac{#1}{#2}\right )}
\newcommand{\GL}{\mathrm{GL}}
\newcommand{\SL}{\mathrm{SL}}
\newcommand{\ord}{\operatorname{ord}}
\newcommand{\lcm}{\operatorname{lcm}}

\newcommand{\sk}{\mid_{\tfrac k2}}
\newcommand{\wh}{M_{k/2}^{wh}\left(\Gamma_0(N),\chi\right)}
\newcommand{\whps}{M_{k/2}^{wh}\left(\Gamma_0\left(N\p^2\right),\chi\right)}
\newcommand{\whpskappa}{M_{\kappa/2}^{wh}\left(\Gamma_0\left(N\p^2\right),\chi\right)}
\newcommand{\hps}{S_{\kappa/2}\left(\Gamma_0\left(N\p^2\right),\chi\right)}
\newcommand{\whone}{M_{k/2}^{wh}\left(\Gamma_1(N)\right)}
\newcommand{\whoneint}{M_{k}^{wh}\left(\Gamma_1(N)\right)}
\newcommand{\Tqs}{T_{\kappa/2,N\p^2}\left(Q^2\right)}
\def\smat#1{\left(\begin{smallmatrix} #1 \end{smallmatrix} \right)}
\def\pmat#1{\left(\begin{matrix} #1 \end{matrix} \right)}

\begin{document}

\title{Congruences satisfied by eta-quotients
}

\author[N.C. Ryan, Z. Scherr, N. Sirolli, S. Treneer]{Nathan C. Ryan \and Zachary Scherr \and Nicol\'as Sirolli \and Stephanie Treneer
}


\maketitle

\begin{abstract}

	The values of the partition function, and more generally the Fou\-rier
	coefficients of many modular forms, are known to satisfy certain congruences.
	Results given by Ahlgren and Ono for the partition function and by Treneer for
	more general Fourier coefficients state the existence of infinitely many families
	of congruences.
	In this article we give an algorithm for computing explicit instances of such
	congruences for eta-quotients.
	We illustrate our method with a few examples.

\end{abstract}

\section{Introduction}\label{sec:intro}

The partition function $p(n)$ gives the number of non-increasing sequences of positive integers that sum to $n$. Ramanujan first discovered that $p(n)$ satisfies
\begin{align*}
p(5n+4)&\equiv 0\pmod{5},\\
p(7n+5)&\equiv 0\pmod{7}, \;\mathrm{and}\\
p(11n+6)&\equiv 0\pmod{11},
\end{align*}
for all positive integers $n$. His work inspired broad interest in other linear congruences for the partition function, and further results were found modulo powers of small primes, especially 5, 7 and 11 (see \cite{ahlgrenono} for references). Ono \cite{ono} showed more generally that infinitely many congruences of the form $p(An+B)\equiv 0\pmod{\ell}$ exist for each prime $\ell\geq 5$. This work was extended by Ahlgren \cite{ahlgren} to composite moduli $m$ coprime to $6$. While these works did not provide direct computation of such congruences, many explicit examples were given by Weaver \cite{weaver} and Johansson \cite{johansson} for prime moduli $\ell$ with $13\leq\ell\leq 31$, and by Yang \cite{yang} for $\ell=37$. All of these results give congruences for arithmetic progressions which lie within the class $-\delta_\ell$ modulo $\ell$, where $\delta_\ell=\frac{\ell^2-1}{24}$.

Here we will focus on congruences like those of the form guaranteed by Ahlgren and Ono in \cite{ahlgrenono}, which lie outside the class $-\delta_\ell$ modulo $\ell$.
More precisely, we consider the following result as stated in \cite{treneer1}.

\begin{theorem}\label{thm:partition}
Let $\p \geq 5$ be a prime.  Then for a positive proportion of the primes $Q\equiv -1 \pmod{576\p^j}$, we have
\[
p\left (\frac{Q^3n+1}{24}\right )\equiv 0 \pmod{\p^j}
\]
for all $n$ coprime to $\p Q$ such that $n\equiv 1 \pmod{24}$ and $\kro{-n}{\p}
\neq \kro{-1}{\p}$.
\end{theorem}

Let $\eta(z)$ denote the Dedekind eta function, which is given by
\[
	\eta(z) = q^{\tfrac1{24}}\prod_{n=1}^\infty \left(1-q^n\right),
	\qquad q=e^{2\pi iz}.
\]
The proof of Theorem~\ref{thm:partition} relies on the relationship between
$p(n)$ and $\eta(z)$, given by
\[
	\sum_{n=-1\pmod{24}}p\left (\frac{n+1}{24}\right)q^n =
	\frac{1}{\eta(24z)}.
\]
The function on the right is an example of an \emph{eta-quotient}, a particular
class of weakly holomorphic modular forms built from $\eta(z)$ (see
Definition~\ref{def:etaquotient} below for details).
This class includes generating functions for a wide variety of partition
functions, including overpartitions, partitions into distinct parts, colored
partitions, and core partitions.

Methods from the theory of modular forms, Hecke operators, and twists by
quadratic characters, are used to derive congruences for the coefficients of
these modular forms. 
Moreover, the fourth author \cite{treneer1} showed that congruences such as
those in Theorem \ref{thm:partition} are common to the coefficients of all
weakly holomorphic modular forms in a wide class.

\begin{theorem}[Theorem 1.2 in \cite{treneer1}]\label{thm:thm12}
Suppose that $\p$ is an odd prime and that $k$ is an odd integer.  Let $N$ be a
positive integer with $4\mid N$ and $(N,\p)=1$.  Let $K$ be a number field with
ring of integers $\OO_K$, and suppose that $f=\sum_n a(n)q^n\in
\whone \cap \OO_K((q))$.  
Let $f$ satisfy Condition C at $\p$.  Then for each positive integer $j$, a
positive proportion of the primes $Q\equiv -1\pmod{N\p^j}$ satisfy the congruence 
\[
a\left(Q^3n\right)\equiv 0\pmod{\p^j}
\]
for all $n \geq 1$ coprime to $\p Q$ such that $\kro{-n}{\p} \neq \vareps_\p$.
\end{theorem}

For a description of Condition C and the numbers $\vareps_\p$, see
Definition~\ref{def:conditionC}.
A similar result holds for integral weight modular forms.

\begin{theorem}[Theorem 1.3 in \cite{treneer1}]\label{thm:thm13}
Suppose that $\p$ is an odd prime and that $k$ is an integer.  Let $N$ be a
positive integer with $(N,\p)=1$. 
Let $K$ be a number field with ring of integers $\OO_K$, and suppose
that $f=\sum_n a(n)q^n\in \whoneint \cap \OO_K((q))$.  
Let $f$ satisfy Condition C at $\p$. Then for each positive integer $j$,
	a positive proportion of the primes $Q\equiv -1 \pmod{N\p^j}$ have the property that 
\[
a(Qn)\equiv 0\pmod{\p^j}
\]
for all $n \geq 1$ coprime to $\p Q$ such that $\kro{-n}{\p} \neq \vareps_\p$.
\end{theorem}

We call a prime $Q$ as in Theorem~\ref{thm:thm12} or Theorem~\ref{thm:thm13} an
\emph{interesting} prime (with respect to $f, \p$ and $j$). 
While these results give the existence of infinitely many interesting primes,
they are not explicitly given, and in general they have not been simply
described. 
A notable exception is the work of Atkin \cite{atkin}, in which it is shown that
every prime $Q\equiv -1\pmod{\ell}$ yields congruences as in Theorem
\ref{thm:partition} when $\ell\in\{5,7,13\}$ and $j=1$.
The main goal of this paper is to present results and algorithms that allow us
to find interesting primes $Q$, and hence prove explicit congruences for the
coefficients of eta-quotients. While we only explicitly state congruences for
select eta-quotients, we remark that our algorithms are general and, with enough
computing power, would work for any eta-quotient.

\medskip

We start by presenting some background on weakly holomorphic modular forms in
general and eta-quotients in particular.  We proceed from there to discuss
various computations we carry out with eta-quotients:  we recall how to compute
coefficients of a general eta-quotient and we discuss how to compute the
expansions of a general eta-quotients at cusps other than $\infty$. 
In the subsequent section, we prove the theorems related to the computation of
interesting primes.
We conclude by describing the algorithms given by these theorems and presenting
some examples.
In particular, we give new explicit congruences for the partition function.

\section{Background and preliminaries}

Given a positive integer $N$ divisible by $4$ and an integer $k$, we let 
$\whone$ denote the space of weakly holomorphic modular forms for
$\Gamma_1(N)$ of weight $k/2$. 
We do not make hypotheses on the parity of $k$, in order to handle both
half-integral and integral weight cases at the same time.
Given a Dirichlet character $\chi$ modulo $N$ we denote by $\wh$ the
corresponding subspace.
We
denote by $G$ the four-fold cover of $\GL_2^+(\QQ)$ acting on these spaces by the
slash operator $\sk$. We refer the reader to \cite{koblitz} for details.

We consider the Fourier expansion of a weakly holomorphic modular form $f \in
\wh$ at a cusp $s=a/c$ in $\QQ\cup \{\infty\}$. 
To such an $s$ we associate an element $\xi_s = (\gamma_s,\phi) \in G$ with
$\gamma_s \in \SL_2(\ZZ)$ such that $\gamma_s \cdot \infty=s$.
Then in \cite[p. 182]{koblitz} it is shown that
\begin{equation}\label{eq:FE}
	f(z) \sk \xi_s 
	= \sum_{n \geq n_0} a_{\xi_s}(n) q_{h_s}^{n+\tfrac{r_s}{4}},
\end{equation}
where $h_s = \tfrac{N}{\gcd(c^2,N)}$ is the width of $s$ as a cusp of
$\Gamma_0(N)$, $r_s \in \{0,1,2,3\}$ depends on $h_s, k$ and $\chi$ (see
\cite[Proposition 2.8]{treneer3} for an explicit formula) and $
q_{h_s}=\exp(2\pi i z/h_s)$.
While the value of the coefficient $a_{\xi_s}(n)$ in the expansion depends on
$\xi_s$, whether or not the coefficient is zero does not depend on $\xi_s$.

The theorems in \cite{treneer1} about congruences mod $\p^j$ are subject to
whether the modular form we are considering satisfies the following
condition for $\p$:
\begin{definition}\label{def:conditionC}
We say that $f$ satisfies \emph{Condition C} for a prime $\p$ if there exists a
sign $\eps = \pm 1$ such that for each cusp $s$, the following is true for
the Fourier expansion in \eqref{eq:FE}: for all $n<0$, with $\p\nmid (4n+r_s)$
and $a_{\xi_s}(n)\neq 0$, we have
\[
\kro{4n+r_s}{\p}= \eps\kro{h_s}{\p}.
\]
\end{definition}
We note that a modular form $f$ can satisfy Condition C for a prime $\p$ in a
trivial way:  namely, if for each cusp $s$ there is no $n<0$ with
$a_{\xi_s}(n)\neq 0$ and $\p\nmid (4n+r_s)$.  
In that case, we set $\eps = 0$.
We also note that when $s=\infty$ we have $r_s=0$ and $h_s=1$, and so Condition
C becomes $\kro{n}{\p}=\eps$ for all $n<0$ with $\p\nmid n$ and $a(n)\neq 0$. 

\medskip

We conclude these preliminaries with the following result regarding the order of
vanishing after twisting, which will be needed later.

\begin{lemma}\label{lem:vanishing_twist}
	Let $f \in \wh$. 
	Let $\p$ be a prime not dividing $N$ and let $\psi$ be a character modulo $\p$.
	Let $s$ be a cusp with respect to $\Gamma_0(N \p^2)$.
	Then 
	\[
		\ord_s(f\otimes \psi) \geq v_0,
	\]
	where $v_0 = \min\{\ord_t(f): t \text{ a cusp with respect to } \Gamma_0(N)\}$.
\end{lemma}

\begin{proof}
	We follow \cite[Lemma 2, p. 127]{koblitz}.
	In \cite[p. 128]{koblitz} it is shown that there exist complex numbers
	$\lambda_\nu$ depending on $\psi$ such that
	\[
		f \otimes \psi = \sum_{\nu = 0}^{\p-1}
		\lambda_\nu \, f_\nu,
	\]
	where $\xi_\nu = \left(\smat{1 & -\nu/\p \\ 0 & 1},1\right)$ and
	$f_\nu = f \sk \xi_\nu$.

	Let $\xi_s =  (\gamma_s,\phi)\in G$, where $\gamma_s \in \SL_2(\ZZ)$ is such
	that $\gamma_s \cdot \infty = s$. 
	Write
	\[
		\pmat{\p & -\nu \\ 0 & \p} \gamma_s = \alpha \pmat{a & b \\ 0 & d},
	\]
	with $\alpha \in \SL_2(\ZZ)$ and $a,d \in \ZZ$.
	Then if we let $t = \alpha \cdot \infty$ and we consider it as a cusp with
	respect to $\Gamma_0(N)$, it is easy to see that $a\in \{1,\p\}$, which since
	$\p \nmid N$ implies that $h_t = \frac ad \, h_s$.

	Let $\sum_{n \geq v_0} a_t(n) q^{n+\tfrac{r_t}4}_{h_t}$
	be a Fourier expansion of $f$ at $t$.
	Then 
	\[
		f_\nu \sk \xi_s =
		\left(a/\p\right)^{k/2}
		e^{\frac{2\pi i b r_t}{4 h_t d}}
		\sum_{n \geq v_0}
		a_t(n) \, 
		e^{\frac{2\pi i n b}{h_t d}}
		q_{h_s}^{n + \tfrac{r_t}4},
	\]
	which shows that $\ord_s \left(f_\nu\right) \geq v_0$ (and that $r_s =
	r_t$).
\end{proof}
\subsection{Eta-quotients}

Recall that the Dedekind eta function is defined by the infinite product
\begin{equation}\label{eq:eta}
\eta(z) = q^{\tfrac{1}{24}} \prod_{n=1}^\infty \left(1-q^n\right).
\end{equation}

\begin{definition}\label{def:etaquotient}
Let $X = \{(\delta,r_\delta)\}$ be a finite subset of $\ZZ_{>0} \times \ZZ$.
Assume that $X$ satisfies
\[
	\sum_X r_\delta \delta \equiv 0 \pmod{24}.
\]
By $\eta^X$ we denote the eta-quotient
\begin{equation}\label{eq:etaq}
	\eta^X(z) = \prod_X \eta(\delta z)^{r_\delta}.
\end{equation}
\end{definition}

It is proved in \cite[Corollary 2.7]{treneer3} that $\eta^X \in \wh$,
where $k = \sum_X r_\delta$, the level $N$ is the smallest multiple of $4$ and
of every $\delta$ such that
\[
	N \sum_X \frac{r_\delta}{\delta} \equiv 0 \pmod{24},
\]
and letting $s = \prod_X\delta^{r_\delta}$  the character
is given by
\[
	\chi(d) = 
	\begin{cases}
		\kro{2s}{d}, & \text{for odd } k, \\
		\kro{4s}{d}, & \text{for even } k. \\
	\end{cases}
\]

For our calculations, we will need the order of vanishing of an eta-quotient
$\eta^X$ at a cusp $s = a/c \in \QQ\cup\{\infty\}$ with respect to the variable
$q_{h_s}$.
Following \cite{ligozat} we have that, if $\gcd(a,c) = 1$, then
\begin{equation}\label{eqn:ligozat}
	\ord_s\left(\eta^X\right) =
	\frac{h_s}{24} \, \sum_X \gcd(c,\delta)^2 \, \frac{r_\delta}{\delta}.
\end{equation}

\section{Computing eta-quotients}\label{sec:computing}

To verify that a weakly holomorphic modular form $f$ satisfies Condition C, we
need to be able to compute the Fourier expansion of $f$ at cusps other than
$\infty$. 
This is addressed by \cite{collins,dickson} in the case of holomorphic forms. 
Rather than extending these general methods to arbitrary weakly holomorphic
modular forms, we limit our attention to forms that are eta-quotients.

For an eta-quotient we can take advantage of the very simple nature of $\eta(z)$
in two different ways.  First, we can quickly compute coefficients of $\eta(z)$
at $\infty$ using Euler's Pentagonal Number Theorem:  
\begin{equation}\label{eq:pentagonal}
\prod_{n=1}^\infty (1-x^n) 
= 1+ \sum_{k=1}^\infty (-1)^k \left(x^{k(3k+1)/2}+x^{k(3k-1)/2}\right).
\end{equation}
Second, we can take advantage of the transformation properties of $\eta(z)$
under the action of $G$ to compute the expansion of $\eta(z)$ at any
cusp $s$ as in \eqref{eq:FE}.  We describe this in more detail now.

Let $\eta^X$ be an eta-quotient as in \eqref{eq:etaq}. Let $s \in \QQ$ be a
cusp, and let $\xi_s = (\gamma_s,\phi) \in G$ be such that $\gamma_s \cdot
\infty=s$. 
Then

\[
	\eta^X(z) \sk \xi_s
	= \prod_X \left(\eta(\delta z)\mid_{\frac12}\xi_s\right)^{r_\delta}
\]
and so our problem reduces to computing $\eta(\delta z)\mid_{\frac12} \xi$ for a
given $\xi = (\gamma,\phi) \in G$, with $\gamma = \smat{a & b \\ c & d} \in
\SL_2(\ZZ)$.
Since the calculations depend only on $\gamma$, by abuse of notation we slash by
matrices of $\GL_2^+(\QQ)$ rather than by elements of $G$.

We have that
\[
	\eta(z)\mid_{\frac12}  \smat{\delta & 0 \\0& 1} = \delta^{\frac14}\eta(\delta z).
\]
Furthermore, we can write
\[
	\smat{\delta & 0 \\0& 1}\gamma =
	\gamma' \smat{A&B\\0&D}
\]
where $\gamma'\in\SL_2(\ZZ), \, A=\gcd(c,\delta), \, D=\tfrac{\delta}{A}$, and
$B$ is an integer so that $\delta \mid (Ad-Bc)$.  
We have that
\[
	\eta(z)\mid_{\frac12}\gamma' = \vareps_{\gamma'} \, \eta(z),
\]
where $\vareps_{\gamma'}$ is 24\textsuperscript{th} root of unity given 
explicitly in terms of $\gamma'$ in \cite[(74.93)]{rademacher}. Then 
\begin{align*}
	\eta(\delta z)\mid_{\frac12} \gamma
	&= \delta^{-\frac14}\eta(z)\mid_{\frac12} \smat{\delta & 0 \\0& 1}\gamma
	=\delta^{-\frac14}\eta(z)\mid_{\frac12}\gamma'
    \smat{ A & B \\0& D}\\ 
			 &= \delta^{-\frac14}\vareps_{\gamma'}\eta(z)\mid_{\frac12}
	\smat{A & B \\0& D}
	= \delta^{-\frac14}\vareps_{\gamma'}
	\left(\tfrac{A}{D}\right)^{\frac14}\eta\left(\tfrac{Az+B}{D}\right)
	=D^{-\frac12}\vareps_{\gamma'}\eta\left(\tfrac{Az+B}{D}\right).
\end{align*}
Using \eqref{eq:eta} we conclude that
\begin{equation}\label{eq:etadelta}
	\eta(\delta z)\mid_{\frac12}\gamma 
	= D^{-\frac12} \, \vareps_{\gamma'} \,
	e^{2\pi i B/(24D)} \, q^{A/(24D)} \,
	\prod_{n=1}^\infty \left(1-e^{2\pi i n B/D}q^{An/D}\right),
\end{equation}
and we compute the infinite product using \eqref{eq:pentagonal}.


Finally, to compute the Fourier expansion of $\eta^X$ at the cusp $s$ we expand
each eta factor as in \eqref{eq:etadelta} and multiply these expansions
together.  

\section{Results}\label{sec:results}

We first recall the Sturm bound for half-integral weight modular forms.
Given a non-negative integer $M$, denote $\mu_M = [\SL_2(\ZZ) : \Gamma_0(M)]$.

\begin{prop}\label{prop:sturm}
	Let $\kappa$ be a positive integer.
	Let $K$ be a number field with ring of integers $\OO_K$, and suppose
	that $h=\sum_{n \geq 1} c(n) q^n \in 
	S_{\kappa/2}\left(\Gamma_0(M),\chi\right) \cap \OO_K[[q]]$,
	with $\chi$ of order $m$.
	Let $\p$ be a prime in $\OO_K$. 
	Let
	\begin{equation}\label{eqn:sturm}
		n_0 = \left\lfloor\frac{\kappa \mu_M}{24} - \frac{\mu_M-1}{4mM}\right\rfloor
			+1.
	\end{equation}
	If $c(n) \equiv 0 \pmod{\p^j}$ for $1 \leq n \leq n_0$, then $h \equiv 0
	\pmod{\p^j}$.
\end{prop}

\begin{proof}
	When $j = 1$ this follows from the standard Sturm bound for cusp forms of
	integral weight and trivial character
	(see \cite[Theorem 9.18]{stein})
	applied to $h^{4m} \in S_{2\kappa m}\left(\Gamma_0(M)\right)$.
	The general case follows by induction on $j$.
\end{proof}

\begin{theorem}\label{thm:main}
	Assume the same hypotheses as in Theorem~\ref{thm:thm12} and that $f$ is an
	eta-quotient. 
	Then for each positive integer $j$ there exist effectively computable
	non-negative integers $\kappa$ and $n_0$ such that, if
	$Q$ is a prime with $Q \equiv -1 \pmod {N\p^j}$ such that 
	\begin{equation}\label{eqn:the_condition}
		a(Q^2 n) +
		\kro{(-1)^{\tfrac{\kappa-1}{2}}n}{Q} Q^{\tfrac{\kappa-3}{2}} a(n) +
		Q^{\kappa-2} a\kro n{Q^2}
		\equiv 0 \pmod {\p^j}
	\end{equation}
	for every $1\leq n \leq n_0$ such that $\p\nmid n$ and $\kro n\p \neq \eps$,
	then $Q$ is an interesting prime.
\end{theorem}

\begin{proof}
	Let $F_\p$ denote the eta-quotient given by
	\[
		F_\p(z) = \begin{cases}
			\tfrac{\eta^{27}(z)}{\eta^3(9z)}   & \text{if } \p = 3, \\
			\tfrac{\eta^{\p^2}(z)}{\eta(\p^2 z)} & \text{if } \p \geq 5.
		\end{cases}
	\]
	We have that $F_\p \in M_{k_\p/2}(\Gamma_0(\p^2))$, where
	\begin{equation}\label{eqn:kp}
		k_\p = \begin{cases}
			24, & \p = 3, \\
			\p^2-1, & \p \geq 5.
			\end{cases}
	\end{equation}
	Furthermore, it satisfies
	\begin{equation}\label{eqn:congruence_Fp}
		{F_\p}^{\p^j-1} \equiv 1 \pmod{\p^j}, \quad j \geq 1.
	\end{equation}

	Consider $F_\p$ as a form with level $\p^2 N$.
	Then if $s = a/c$ is a cusp with $c \mid \p^2 N$ and $\p^2 \nmid c$, by
	\eqref{eqn:ligozat} the order of vanishing of $F_\p$ at $s$ is given by
	\begin{equation}\label{eqn:vanishing_Fp}
	    \ord_s(F_\p) = 
		\frac{N}{\gcd(c^2,N)} \cdot
		\begin{cases}
            10   & \text{if } \p = 3 \text{ and } \p \nmid c, \\
            1   & \text{if }  \p = 3 \text{ and } \p \mid c, \\
			\tfrac{\p^4-1}{24} & \text{if } \p \geq 5 \text{ and } \p\nmid c, \\
			\tfrac{\p^2-1}{24} & \text{if } \p \geq 5 \text{ and } \p\mid c.
		\end{cases}
	\end{equation}

	Let $\chi$ denote the character of $f$.
	We consider
	\[
		f_\p = f \otimes \chi_\p^{triv} - \eps f \otimes \kro{\cdot}\p
		\quad \in \whps,
	\]
	where $\eps$ is given by Condition C, and $\chi_\p^{triv}$
	denotes the trivial character modulo $\p$.
	Then $f_\p$ vanishes at every cusp $s = a/c$ 
	with $c \mid \p^2N$ and $\p^2 \mid c$
	(see \cite[Proposition 3.4]{treneer1}).

	Let $v_0$ be the integer given by $f$ following Lemma
	\ref{lem:vanishing_twist}.
	This number can be computed using \eqref{eqn:ligozat}, since $f$ is an
	eta-quotient.
	Using \eqref{eqn:vanishing_Fp}, we compute a (small as possible) integer
	$\beta$ such that $\beta \geq j-1$ and
	\[
		\p^\beta \ord_s(F_\p) > -v_0
	\]
	for every cusp $s = a/c$ such that $c \mid \p^2 N$ and
	$\p^2 \nmid c$.
	Furthermore, we let $\kappa = k + \p^\beta k_\p$, and we assume that $\beta$ is
	such that $\kappa > 0$.
	We consider
	\[
		g_{\p,j} = \tfrac12 \, f_\p \cdot  {F_\p}^{\p^\beta}
		\quad \in \whpskappa.
	\]
	Then $g_{\p,j}$ vanishes at every cusp.
	Furthermore, by \eqref{eqn:congruence_Fp} it satisfies that
	\begin{equation}\label{eqn:congruence_hp}
		g_{\p,j} \equiv \sum_{\p \nmid n,\,\kro n\p \neq \eps} a(n) q^n \pmod{\p^j}.
	\end{equation}
	
	Let $Q$ be a prime such that $Q \equiv -1 \pmod{N \p^j}$, and let
	\[
		h_{\p,j} = g_{\p,j} \vert \Tqs
	    \quad \in \hps.
	\]
	Write $h_{\p,j} = \sum_{n=1}^\infty c(n) q^n$.
	The formulas for the action of $\Tqs$ in terms of
	Fourier coefficients imply that if $h_{\p,j} \equiv 0 \pmod{\p^j}$, then $Q$
	is an interesting prime (see the the proof of \cite[Theorem 1.2]{treneer1}
	for details).

	Moreover, by these formulas	and \eqref{eqn:congruence_hp} we have that $c(n)
	\equiv 0 \pmod{\p^j}$ if $\p \mid n$ or if $\kro n\p = \eps$, and otherwise
	$c(n)$ equals the left hand side of \eqref{eqn:the_condition}, modulo
	$\p^j$.
	Hence the result follows from Proposition~\ref{prop:sturm}, taking
	$n_0$ as in \eqref{eqn:sturm} with $M = N\p^2$.
\end{proof}

We now consider the integral weight case.

\begin{theorem}
	Assume the same hypotheses as in Theorem~\ref{thm:thm13} and that $f$ is an
	eta-quotient.  
	Then for each positive integer $j$ there exist effectively computable
	non-negative integers $\kappa$ and $n_0$ such that, if
	$Q$ is a prime with $Q \equiv -1 \pmod {N\p^j}$ such that
	\begin{equation}\label{eqn:the_condition13}
		a(Q n) + Q^{\kappa-1} a\kro n{Q}
		\equiv 0 \pmod {\p^j}
	\end{equation}
	for every $1\leq n \leq n_0$ such that $\p \nmid n$ and $\kro n\p \neq \eps$,
	then $Q$ is an interesting prime.
\end{theorem}

\begin{proof}
	We replace $N$ by $\lcm(4,N)$ and consider $f$ as a form of level $N$ and
	weight $2k/2$ in order to be able to use \cite[Theorem 3.1]{treneer1}.
	We let $g_{\p,j}$ and $h_{\p,j}$ be as
	in the previous proof, so that if $h_{\p,j} \equiv 0 \pmod{\p^j}$ then $Q$ is
	an interesting prime.
	Write $h_{\p,j} = \sum_{n=1}^\infty c(n) q^n$. Then if $\p \mid n$ or $\kro n\p
	=\eps$ we have that $c(n) \equiv 0 \pmod{\p^j}$, and otherwise $c(n)$
	equals the left hand side of \eqref{eqn:the_condition13} modulo $\p^j$.
	The	result follows then from Proposition~\ref{prop:sturm}.
\end{proof}

\begin{rmk}
	In the proof of Theorems \ref{thm:thm12} and \ref{thm:thm13}
	the existence of interesting primes among the candidates 
	(i.e., those $Q$ such that $Q\equiv -1 \pmod {N\p^j}$) is given as
	a consequence of Chebotarev's density theorem. 
	Thus we expect, asymptotically, to have them evenly distributed.
	However, in the examples we computed in the next section we obtained that
	most of the candidates that are computationally
	tractable are actually interesting.
\end{rmk}

\section{Algorithms and examples}\label{sec:implementation}

We present two algorithms in this section.
First, we describe how to tell if a given eta-quotient satisfies Condition C at
a given prime $\p$ and, second, following the proof of Theorem \ref{thm:main},
we describe an algorithm for finding interesting
primes given an eta-quotient $f$ and a prime $\p$ at which $f$ satisfies Condition C
(and the corresponding value of $\eps$).

These algorithms were implemented in \texttt{Sage} (\cite{sagemath}). Our code
is available at \cite{code}, together with some examples and related data.

\begin{algorithm}\label{alg:conditionC}
	\LinesNumbered
	\KwData{An eta-quotient $f$ of level $N$; an odd prime $\p$; $\eps \in
	\{\pm1 ,0\}$}
	\KwResult{True if Condition C is satisfied at $\p$ with $\eps$, and
	False otherwise}
    \BlankLine
	\For{every cusp $s$ in $\Gamma_0(N)$}
		{
			$n_0 \leftarrow \ord_s(f)$, using \eqref{eqn:ligozat}\;
			\For{$n \leftarrow -1$ \KwTo $n_0$}
			{
				$c \leftarrow a_{\xi_s}(n)$, computed following Section
				\ref{sec:computing}\;
				\If{$c \neq 0$}
				{
					\If{$\kro{4n+r_s}{\p} \neq \eps \kro{h_s}{\p}$}
						{\Return False}
				}
			}
		}
	\Return True
    \BlankLine

	\caption{Algorithm for deciding if Condition C is satisfied.}
\end{algorithm}

\begin{algorithm}\label{alg:interesting}
	\LinesNumbered
 \KwData{An eta-quotient $f=\sum a(n)q^n$ of level $N$ and half-integral weight
	 $k/2$; an odd prime $\p$ at which condition C is satisfied and the
	 corresponding $\eps$;
	 an integer $j\geq 1$; a candidate $Q$}
	 \KwResult{True if $Q$ is interesting, and False if
	 \eqref{eqn:the_condition} does not hold}
 \BlankLine
$v_0 \leftarrow \min\{\ord_t(f): t \text{ a cusp with respect to }
	\Gamma_0(N)\}$, using \eqref{eqn:ligozat};

 \For{every cusp $s = a/c$ in $\Gamma_0(N\p^2)$ with $c \mid N\p^2$}{
  \If {$\p^2 \nmid c$}{

       $\ord_s(F_\p) \leftarrow$ the result of \eqref{eqn:vanishing_Fp}\;
       $\log_\p \leftarrow \left \lceil \log_\p\left(\frac{-v_0}
	   {\ord_s(F_\p)}\right)\right \rceil$\;

       $\beta \leftarrow \max\{\beta,\log_\p\}$\;

}

}

$\kappa\leftarrow k + \p^\beta k_\p$, where $k_\p$ is given by \eqref{eqn:kp}\;

$n_0 \leftarrow$ Sturm bound for $\hps$, as in Proposition \ref{prop:sturm}\;

interesting $\leftarrow$ True\;
\While{interesting}{
	\For{$n\leftarrow 1$ \KwTo $n_0$}{\If{$\p \nmid n$ and $\kro n\p \neq \eps$}{
\If{$a(Q^2 n) +	\kro{(-1)^{\tfrac{\kappa-1}{2}}n}{Q} Q^{\tfrac{\kappa-3}{2}} a(n) + Q^{\kappa-2} a\kro n{Q^2}\not\equiv 0 \pmod {\p^j}$\label{step:cong}}
{interesting$\leftarrow$ False\;}
}
}
}
\Return interesting\;
\BlankLine

\caption{Algorithm for finding interesting primes.}
\end{algorithm}

\begin{rmk}
	Algorithm \ref{alg:interesting} works in the integral weight case,
	replacing Step \ref{step:cong} according to \eqref{eqn:the_condition13}.
\end{rmk}

\begin{rmk}\label{rmk:fewer_coeffs}

	In general, Step \ref{step:cong} of Algorithm \ref{alg:interesting} requires
	the computation of $a(n)$ up to the product of $Q^2$ and the Sturm bound.  
	For many forms of large level and large $Q$ this can be prohibitive. 
	Strictly speaking, though, the number of coefficients that we need to know
	is roughly equal to the Sturm bound, since
	Step \ref{step:cong} requires computing $a(Q^2n), a(n)$ and $a(n/Q^2)$ for
	$n$ such that $\p \nmid n$ and $\kro n\p \neq \eps$, up to the Sturm bound; the
	$a(n/Q^2)$'s are included in the $a(n)$'s, and the Kronecker symbol
	condition cuts the required number of coefficients in half.
	So, if there is some way to compute these coefficients other than
	multiplying out the eta-quotient \eqref{eq:etaq}, we can use Algorithm
	\ref{alg:interesting} to compute with forms of moderately large Sturm bound
	and for moderately large candidates.
\end{rmk}

\begin{rmk}
	Given a candidate $Q$, if the output of Algorithm \ref{alg:interesting} is
	False, then there exists an $n$ such that \eqref{eqn:the_condition} (or
	\eqref{eqn:the_condition13}, in the integral weight case) does not hold.
	A priori, this does not imply that the prime $Q$ is not interesting.
	Nevertheless in the following examples, at least when it was computationally
	feasible, for each $Q$ for which the Algorithm \ref{alg:interesting}
	returned False, we verified that one of the congruences Theorem
	\ref{thm:thm12} (or Theorem \ref{thm:thm13}, in the integral weight case)
	does not hold.
\end{rmk}

\begin{rmk} To the best of our knowledge, the examples computed below are new explicit congruences for these partition functions, except where otherwise noted.
\end{rmk}

\subsection{Examples of half-integral weight}

\subsubsection*{The partition function}

Recall that an eta-quotient of particular interest is related to the
partition function:
\[
	\sum_{n \equiv -1 \pmod{24}} 
	p \left(\frac{n+1}{24}\right) q^n =
	\frac{1}{\eta(24z)}
	\quad \in M_{-1/2}^{wh}(\Gamma_0(576),\chi).
\]

The first prime at which it satisfies condition $C$ is $\p=5$, with $\vareps_5 =
1$.
Using the notation from Section~\ref{sec:results}, we can take $\beta=0$.  
Then since $\mu_{576\cdot 5^2} = 34560$, the Sturm bound is
\[
	\left\lfloor
		\frac{23}{24} \cdot 34560 -
		\frac{34560-1}{8 \cdot 576\cdot 5^2}
	\right\rfloor
	= 33119.
\]
The first $Q$ such that $Q\equiv -1\pmod{576\cdot 5}$ is $Q=2879$. 
In order to run Algorithm \ref{alg:interesting}, \textit{a priori}, we need to
compute $p\left(\frac{n+1}{24}\right)\pmod{5}$ for $1\leq n \leq 33119\cdot
2879^2 \approx 10^{12}$.
That is beyond the scope of our computational abilities.

Following Remark \ref{rmk:fewer_coeffs}, to check if $Q=2879$ is interesting,
since $Q^2 > 33119$
we only need to calculate
\begin{multline*}
	p\left(\frac{n+1}{24}\right) \quad \text{ and } \quad
	p\left(\frac{Q^2n+1}{24}\right) \pmod5, \\
	\kro n5 = -1, \quad 1 \leq n \leq 33119, 
\end{multline*}
Moreover, we only need to consider $n\equiv -1\pmod{24}$.
In conclusion, we need to compute
\begin{multline*}
	p\left(\frac{n+1}{24}\right) \quad\text{ and }\quad
	p\left(\frac{Q^2n+1}{24}\right) \pmod5, \\
	n\equiv 23, 47 \pmod{120},\quad 1\leq n \leq 33119.
\end{multline*}
These individual coefficients can be computed independently (and in parallel)
using fast algorithms for computing partition numbers.

Similar computations for other candidates $Q$ and other
primes $\p$ at which Condition C is satisfied can also be done.
We summarize the results of our computations in Proposition~\ref{prop:eta-congruences}. 
We computed the needed coefficients using the \texttt{FLINT} library
(\cite{flint}).
The code is available at \cite{code}.

\begin{prop}\label{prop:eta-congruences} 

	We have that
	\[
		p\left(\frac{Q^3n+1}{24}\right)\equiv 0 \pmod{\p},
	\]
	for all $n$ coprime to $\p Q$ such that
	$n \equiv 1 \pmod{24}$ and $\kro{-n}\p \neq \eps$, for

	\begin{itemize}

		\item $\p = 5, \varepsilon_5 = 1$ and
			\[
				Q = 2879, 11519, 23039, 25919, 51839, 66239, 69119, 71999, 86399,
			97919.
			\]
		\item $\p = 7, \varepsilon_7 = -1$ and
			\[
				Q = 16127, 44351, 48383, 68543, 76607.
			\]
		\item $\p = 11, \varepsilon_{11} = -1$ and
			\[
				Q=25343.
			\]
		\item $\p = 13, \varepsilon_{13} = 1$ and
			\[
			Q = 7487, 44927, 67391.
			\]
		\item $\p = 17, \varepsilon_{17} = 1$ and
			\[
			Q = 9791.
			\]
	\end{itemize}
	In each case, these $Q$ are all the interesting primes less than $10^5$.

\end{prop}


\begin{rmk}
	The congruences for $\ell\in\{5, 7,13\}$ were known by Atkin.
	Moreover, \cite[Theorem 2]{atkin} implies that in these cases every
	candidate is interesting.    This agrees with our computations as summarized in Proposition~\ref{prop:eta-congruences}.

In the same article, Atkin also suggests that the same should hold for $\ell \in
\{11,17,19,23\}$; namely, that every candidate is indeed interesting.  
We found counterexamples for each of these $\p$'s.
More precisely, our computations showed that the primes $Q = 12671, 44351, 76031$
are candidates for $\p=11$, but they are not interesting.  
The same holds for the primes $\p = 17$ and $Q = 19583, 68543,
97919$, for $\p = 19$ and $Q = 32831$, for $\p = 23$ and $Q = 66239$, and for
$\p = 29$ and $Q = 16703, 50111$.
\end{rmk}

For each pair of primes $\ell$ and $Q$ above, fixing a suitable residue class for $n$ modulo $24\ell$ will yield a Ramanujan-type congruence. For example, Proposition~\ref{prop:eta-congruences} implies that
\begin{multline*}
	p\left(15956222111407 \, m + 9425121394238 \right)\equiv 0\pmod{17}\, \\
	\mbox{ for }\,m\not\equiv 4007\pmod{9791}.
\end{multline*}

\subsubsection*{Overpartition function}

	Let $\overline p(n)$ denote the overpartition function.
	For information on $\overline p(n)$ see \cite{overpartitions}.  
	It is known that
	\[
		\sum_{n\geq 0} \overline{p}(n)q^n =
		\frac{\eta(2z)}{\eta^2(z)}
		\quad \in M_{-1/2}^{wh}(\Gamma_0(16),\chi).
	\]
	Since we do not have an algorithm for computing the numbers $\overline p(n)$
	individually, in this case we must compute the whole Fourier expansion up to
	the Sturm bound, thus obtaining smaller interesting primes.

\begin{prop}

	We have that
	\[
		\overline{p}\left(Q^3n\right)\equiv 0 \pmod{\p},
	\]
	for all $n$ coprime to $\p Q$ such that $\kro{-n}{\p} \neq \eps$, for

	\begin{itemize}

		\item $\p = 3, \varepsilon_3 = -1$ and
			\[
				Q = 47, 191, 239, 383, 431, 479, 719, 863, 911, 1103.
			\]
		\item $\p = 5, \varepsilon_5 = 1$ and
			\[
				Q = 79, 239, 479.
			\]
	\end{itemize}

\end{prop}

In other words for, $\p = 3$ Algorithm \ref{alg:interesting}
returns true for every candidate $Q < 1151$; the same holds for $\p = 5$ and $Q
< 719$.
On the other hand, for $\p = 7$, it returns False for $Q = 223$.

\subsection{Examples of integral weight}

\subsubsection*{$24$-color partitions}  

Let $p_{24}(n)$ denote the number of 24-color partitions of $n$.  
See \cite{keith} for background on $k$-color partitions.
Let $\Delta\in S_{12}(\SL_2(\ZZ))$ be the normalized cuspform of weight 12 and
level 1. Then
\[
	\sum_{n\geq -1} p_{24}(n+1) q^n =
	\frac{1}{\Delta(z)} =
	\frac{1}{\eta(z)^{24}}
	\quad \in M^{wh}_{-12}\left(\SL_2(\ZZ)\right).
\]

\begin{prop}
	We have that
	\[
	p_{24}\left(Q(n+1)\right) \equiv 0\pmod{\p}
	\]
	for all $n$ coprime to $\p Q$ such that $\kro{-n}{\p} \neq \eps$, for
	every $\p < 12$, with
	\[
		\varepsilon_3 = -1,
		\varepsilon_5 = 1,
		\varepsilon_7 = -1,
		\varepsilon_{11} = -1,
	\]
	and every $Q$ such that $Q \equiv -1 \pmod{\p}$ and $Q < 10^5$.

	The same holds for $\p = 13,\, \varepsilon_{13} = 1$ and
	\[
		Q = 1741, 2963, 4523, 5407, 5563, 5927, 5953, 6733, 7331, 9749.
	\]
\end{prop}

We remark that the remaining candidates $Q < 10^4$ for $\p =13$
(i.e. $Q = 103, 181, 233,\dots$) are not interesting.

\subsubsection*{$3$-core partitions}

Let $B_3(n)$ denote the number of triples of $3$-core partitions of $n$.
In \cite{Wang} it was shown that
\[
	\sum_{n\geq 1} B_3(n-1)q^n = 
	\frac{\eta(3z)^9}{\eta(z)^3}
	\quad \in M_3(\Gamma_0(3)).
\]
Since this is a holomorphic eta-quotient, Condition C holds trivially at every
odd prime.
In this case all the candidates we computed systematically turned out to be
interesting. 
We were not able to find a candidate for which Algorithm \ref{alg:interesting}
returns False.

\begin{prop}
	
	We have that
	\[
	B_3\left(Q(n-1)\right)\equiv 0 \pmod{\p}
	\]
	for every $n$ coprime to $\p Q$,
	for every $\p<14$ and for every $Q\equiv -1\pmod{3\p}$ less than $10^4$.

\end{prop}



\end{document}